\documentclass[12pt,a4paper,american]{article}
\usepackage[T1]{fontenc}
\usepackage[latin9]{inputenc}
\usepackage{amsthm}
\usepackage{amsmath}
\usepackage{amssymb}
\usepackage{graphicx}

\usepackage{multirow}

\makeatletter
\newcommand{\lyxaddress}[1]{
\par {\raggedright #1
\vspace{1.4em}
\noindent\par}
}
\theoremstyle{plain}
\newtheorem{thm}{Theorem}
  \theoremstyle{definition}
  
  \theoremstyle{remark}
  \newtheorem{rem}[thm]{Remark}
  \theoremstyle{plain}
  \newtheorem{prop}[thm]{Proposition}
  \theoremstyle{plain}
  
  \theoremstyle{plain}
  \newtheorem{cor}[thm]{Corollary}
 \theoremstyle{definition}
  
  \theoremstyle{remark}
  \newtheorem*{rem*}{Remark}

  \theoremstyle{definition}

\usepackage{mathrsfs}

\theoremstyle{plain}

\renewcommand{\Im}{\mathop\mathrm{Im}\nolimits}

\newcommand{\Res}{\mathop\mathrm{Res}\nolimits}

\makeatother

\usepackage{babel}

\begin{document}

\title{Nevanlinna extremal measures for polynomials related to $q^{-1}$-Fibonacci polynomials}
\author{F.~{\v S}tampach$^{1}$}

\date{{}}

\maketitle

\lyxaddress{$^{1}$Department of Applied Mathematics, Faculty of Information Technology, Czech Technical University in~Prague,  Th{\' a}kurova~9, 160~00 Praha, Czech Republic}

\begin{abstract}
  \noindent The aim of this paper is the study of $q^{-1}$-Fibonacci polynomials with $0<q<1$.
  First, the $q^{-1}$-Fibonacci polynomials are related to a $q$-exponential function which allows an asymptotic analysis to be worked out. Second, related basic 
  orthogonal polynomials are investigated with the emphasis on their orthogonality properties. In particular, a compact formula for the reproducing kernel is obtained that allows to describe 
  all the N-extremal measures of orthogonality in terms of basic hypergeometric functions and their zeros. Two special cases involving $q$-sine and $q$-cosine are discussed in more detail. 
\end{abstract}
\vskip\baselineskip\noindent

\vskip0.5\baselineskip\noindent\emph{2010 Mathematical Subject
Classification}: 11B39, 33D45, 30E05

\section{Introduction}

In \cite{Carlitz75}, Carlitz introduced the $q$-Fibonacci polynomials $\varphi_{n}(x;q)$ by 
\begin{equation}
 \varphi_{n}(x;q)=\sum_{2k<n}{n-k-1\brack k}_{q}q^{k^{2}}x^{n-2k-1}, \label{eq:def_varphi_n}
\end{equation}
for $n\in\mathbb{Z}_{+}$ (nonnegative integers). The $q$-binomial coefficient is defined by
\[
{n \brack k}_{q}=\frac{(q;q)_{n}}{(q;q)_{k}(q;q)_{n-k}},
\]
for $k=0,1,\dots,n$ and $n\in\mathbb{Z}_{+}$, where 
\[
(a;q)_{n}=\prod_{k=0}^{n-1}\left(1-aq^{k}\right)
\]
is the standard notation for the $q$-Pochhammer symbol. 
In this paper, we intensively use the theory of basic hypergeometric series and, in the notation, we will strictly follow Gasper and Rahman's book \cite{GasperRahman}. 
In addition, if it is not stated otherwise, we always assume $0<q<1$.

The polynomials $\varphi_{n}(x;q)$ satisfy the second-order recurrence
\begin{equation}
\varphi_{n+1}(x;q)=x\varphi_{n}(x;q)+q^{n-1}\varphi_{n-1}(x;q), \quad n\in\mathbb{N}, \label{eq:q-Fib_recur}
\end{equation}
with the initial conditions $\varphi_{0}(x;q)=0$ and $\varphi_{1}(x;q)=1$. Let us mention that $\varphi_{n}(1;q)$ are polynomials in $q$ first considered by I.~Schur in conjunction with his proof of the Rogers-Ramanujan
identities \cite{Schur}. It can be referred as the $q$-analogue to the Fibonacci numbers $F_{n}$ since clearly $\varphi_{n}(1;1)=F_{n}$, see \cite{Carlitz74}. For a different definition of $q$-Fibonacci numbers, see 
also \cite{Ismail09}.

One of our primary intention is the study of orthogonality of related polynomials
\begin{equation}
 T_{n}(x;q)=(-{\rm i})^{n}q^{-n/2}\varphi_{n+1}({\rm i}q^{1/2}x;q), \quad n=-1,0,1,2,\dots. \label{eq:def_T_n}
\end{equation}
They fulfill the second-order difference equation 
\begin{equation}
T_{n+1}(x;q)=xT_{n}(x;q)-q^{n-1}T_{n-1}(x;q), \quad n\in\mathbb{Z}_{+}, \label{eq:gener_monic_recur}
\end{equation}
with the initial conditions $T_{-1}(x;q)=0$ and $T_{0}(x;q)=1$. Identity (\ref{eq:def_varphi_n}) together with (\ref{eq:def_T_n}) yield the explicit expression
$$T_{n}(x;q)=\sum_{2k\leq n}{n-k \brack k}_{q}(-1)^{k}q^{k(k-1)}x^{n-2k}, \quad n\in\mathbb{Z}_{+}.$$

The polynomials $\varphi_{n}(x;q)$ as well as $T_{n}(x;q)$ attract attention of mathematicians for their beautiful properties. Let us mention at least their intriguing 
relation to the Rogers-Ramanujan function, associated continued fraction, and the celebrated Rogers-Ramanujan identities \cite{Al-SalamIsmail, Andrews70, Andrews04, Cigler04, Ismail-book}. 
Some algebraic and combinatorial properties of these polynomials have also been studied in \cite{Carlitz75, Cigler03}.

By the spectral theorem for orthogonal polynomials, the sequence $\{T_{n}(x;q)\}_{n=0}^{\infty}$ obeying the recurrence rule (\ref{eq:gener_monic_recur}), 
constitutes a system of polynomials orthogonal with respect to a positive Borel measure on $\mathbb{R}$ for any $q>0$.
In the case $0<q<1$, polynomials $T_{n}(x;q)$ are a~particular case of orthogonal polynomials studied by Al-Salam and Ismail 
in \cite{Al-SalamIsmail}, consult also \cite[Sec.~13.6]{Ismail-book}. The respective measure of orthogonality $\mu_{q}$ is unique and its Stieltjes transform can be 
expressed as a generalized Rogers-Ramanujan continued fraction.

Concerning the asymptotic behavior of $T_{n}(x;q)$, for $n\rightarrow\infty$ and $0<q<1$, one has
\[
 \lim_{n\rightarrow\infty}x^{-n}T_{n}(x;q)=A_{q}\left(q^{-1}x^{2}\right) \quad \forall x\neq0,
\]
where
\[
 A_{q}(z)=\,_{0}\phi_{1}\!\left(\,;0;q,-qz\right)
\]
is the Ramanujan's entire function. 
With the aid of Markov's theorem, one can express the Stieltjes transform of $\mu_{q}$ in the form
\begin{equation}
 \int_{\mathbb{R}}\frac{\mbox{d}\mu_{q}(x)}{1-xz}=\frac{A_{q}\left(z^{2}\right)}{A_{q}\left(q^{-1}z^{2}\right)},
\label{eq:stielt_trans_eq_ratio_RR_func}\end{equation}
see \cite[Thm.~3.2,~3.4]{Al-SalamIsmail}. Equality (\ref{eq:stielt_trans_eq_ratio_RR_func}) 
holds for all $z\in\mathbb{C}$ for which $A_{q}\left(q^{-1}z^{2}\right)\neq0$. Recall that the set of zeros of the
Ramanujan's entire function $A_{q}$ is composed of denumerable number of positive isolated points accumulating at $\infty$, \cite{Al-SalamIsmail, Ismail05}.
Let us denote these zeros as $z_{j}=z_{j}(q)$, for $j\in\mathbb{N}$. The orthogonality relation for $\{T_{n}(x;q)\}_{n=0}^{\infty}$ reads
\[
 \sum_{j=1}^{\infty}\frac{A_{q}(qz_{j})}{z_{j}A_{q}'(z_{j})}T_{n}\left(\pm q^{-1/2}z_{j}^{-1/2}\right)T_{m}\left(\pm q^{-1/2}z_{j}^{-1/2}\right)=-2q^{n(n-1)/2}\delta_{mn},
\]
where $T_{n}(\pm y)T_{m}(\pm y)$ is a shorthand for the expression
\[
 T_{n}\left(y;q\right)T_{m}\left(y;q\right)+T_{n}\left(-y;q\right)T_{m}\left(-y;q\right)\!,
\]
see \cite[Thm.~3.3, Cor.~4.5]{Al-SalamIsmail}.

In the case $q=1$, the polynomials $T_{n}(x;1)$ are essentially the Chebyshev polynomials of the second kind $U_{n}$. More precisely, one has $T_{n}(x;1)=U_{n}\left(x/2\right)$. 
This family of polynomials is deeply investigated and very well known \cite{Chihara}.

In this paper we focus on the case $q>1$. This case is particularly interesting for there are infinitely many measures
of orthogonality. In other words, the corresponding Hamburger moment problem is in the indeterminate case if $q>1$. 
Solutions of the indeterminate moment problem are characterized with the aid of four entire functions used in the Nevanlinna parametrization. 
In the case under investigation, these functions are expressed as certain $q$-extensions 
of sine $\mathcal{S}_{q}$ and cosine $\mathcal{C}_{q}$, for the functions $\mathcal{S}_{q}$ and 
$\mathcal{C}_{q}$ naturally arise in the asymptotic formulas for polynomials $T_{n}(x;q)$, as $n\rightarrow\infty$.

This paper is organized as follows. In Section \ref{sec:func_EqSqCq}, we recall certain version of $q$-exponential function $\mathcal{E}_{q}$, associated $q$-trigonometric functions 
$\mathcal{S}_{q}$ and $\mathcal{C}_{q}$, and their hyperbolic counterparts $\mathcal{S}h_{q}$ and $\mathcal{C}h_{q}$. These functions appeared at many various places and have been 
studied extensively in past, though not in the context of $q$-Fibonacci polynomials. Section \ref{sec:func_EqSqCq} recalls their basic properties and above all an addition 
formula for the product of two $q$-exponential functions which is an essential identity needed for derivation of main results of this paper.

Section \ref{sec:rel_qE_eFib} is devoted to the study of relations between the $q$-exponential $\mathcal{E}_{q}$ and the $q^{-1}$-Fibonacci polynomials $\varphi_{n}(x;q^{-1})$.
We consider $q^{-1}$-Fibonacci polynomials with $0<q<1$ which is the same, though in some cases more convenient, as considering $q$-Fibonacci polynomials with $q>1$.
First of all, we derive a formula for $q^{-1}$-Fibonacci polynomials in terms of the $q$-exponential $\mathcal{E}_{q}$. As a direct consequence, we obtain asymptotic formulas 
for $\varphi_{n}(x;q^{-1})$, as $n\rightarrow\infty$. Depending on the parity of the index $n$, the functions  $\mathcal{S}h_{q}$ and $\mathcal{C}h_{q}$ arise as the leading 
term in the asymptotic formulas. Further, we prove a $q$-version of the Euler--Cassini formula for $q^{-1}$-Fibonacci polynomials. We use a method based on a very general identity 
which could be used analogously in numerous similar situations. As an application, we obtain several more subtle identities between
$q^{-1}$-Fibonacci polynomials and the functions $\mathcal{E}_{q}$, $\mathcal{S}h_{q}$, and $\mathcal{C}h_{q}$.

In Section \ref{sec:OPs_repro_ker_Next_meas}, we introduce a family of orthogonal polynomials $P_{n}(x;q)$ related to $q^{-1}$-Fibonacci polynomials. 
We derive limit relations for $P_{n}(x;q)$, as $n\rightarrow\infty$, as a straightforward consequence of the knowledge of the asymptotic behavior of $q^{-1}$-Fibonacci polynomials.
This provides an alternative way of asymptotic analysis of $P_{n}(x;q)$ to the usual method based on an appropriate generating
function formula for $P_{n}(x;q)$ and Darboux's theorem as it was originally done by Chen and Ismail in \cite[Sec.~3]{ChenIsmail}.

The second part of Section \ref{sec:OPs_repro_ker_Next_meas} contains the most important results concerning orthogonality of the polynomials $P_{n}(x;q)$. First, we recall that functions 
from the Nevanlinna parametrization can be expressed in terms of the functions $\mathcal{S}_{q}$ and $\mathcal{C}_{q}$.
Next, with the aid of the addition formula from Section \ref{sec:func_EqSqCq}, we find a compact expression 
for the reproducing kernel for the polynomials $P_{n}(x;q)$. These two ingredients allow us to describe all N-extremal measures in a closed form and write down the corresponding orthogonality relations for $P_{n}(x;q)$. 
In addition, we point out two special cases of orthogonality relations written in terms of the $q$-sine $\mathcal{S}_{q}$ and its zeros only, and similar relation with the $q$-cosine $\mathcal{C}_{q}$ and its zeros.

\section{Functions $\mathcal{E}_{q}$, $\mathcal{S}_{q}$ and $\mathcal{C}_{q}$} \label{sec:func_EqSqCq}

There are several known $q$-deformations of the exponential 
function. Two commonly known are due to Jackson:
\[
E_{q}(z)=\sum_{n=0}^{\infty}\frac{q^{n(n-1)/2}}{(q;q)_{n}}z^{n} \quad (z\in\mathbb{C}) \quad \mbox{ and } \quad
e_{q}(z)=\sum_{n=0}^{\infty}\frac{z^{n}}{(q;q)_{n}} \quad (|z|<1),
\]
see \cite[\S1.3]{GasperRahman}.

The one-parameter generalization 
\[
 E_{q}^{(\alpha)}(z)=\sum_{n=0}^{\infty}\frac{q^{\alpha n^{2}/2}}{(q;q)_{n}}z^{n}
\]
where $\alpha\geq0$, was studied by Atakishieyev in \cite{Atakishiyev}, see also references therein.
Obviously, one has $E_{q}^{(0)}(z)=e_{q}(z)$ and $E_{q}^{(1)}(z)=E_{q}(q^{1/2}z)$. 
Ismail shows in \cite[Sec.~14.1]{Ismail-book} that $E_{q}^{(\alpha)}$ are entire functions of order zero,
hence they have infinitely many zeros. The particular case corresponding to the value $\alpha=1/2$ plays a crucial 
role in asymptotic formulas for $q^{-1}$-Fibonacci polynomials, as $n\rightarrow\infty$. An asymptotic and numerical
study of zeros of $E_{q}^{(1/2)}$ has been made in \cite{NelsonGartley}.
For the sake of simplicity, we use the notation $\mathcal{E}_{q}(z)$ for $E_{q}^{(1/2)}(z)$. Thus, we consider the function
\[
 \mathcal{E}_{q}(z)=E_{q}^{(1/2)}(z)={}_{1}\phi_{1}\left(0;-q^{1/2};q^{1/2},-q^{1/4}z\right)=\sum_{n=0}^{\infty}\frac{q^{n^{2}/4}}{(q;q)_{n}}z^{n}, \quad z\in\mathbb{C}.
\]
Let us also remark that this function can be obtained as a rescaled limit of yet another more general family of $q$-exponential functions introduced by Ismail and Zhang in \cite{IsmailZhang}
and later extensively studied mainly by Suslov, see \cite{Suslov-book}. Particularly, the function $\mathcal{E}_{q}$ as well as the corresponding trigonometric functions appears in Suslov's book
in \cite[Sec.~2.5]{Suslov-book}. The function $\mathcal{E}_{q}$ fulfills the second order $q$-difference equation
\begin{equation}
  \mathcal{E}_{q}(z)-\mathcal{E}_{q}(qz)=zq^{1/4}\mathcal{E}_{q}(q^{1/2}z).
\label{eq:E_q_recur}\end{equation}

Corresponding $q$-sine and $q$-cosine functions can be introduced by formulas
\begin{equation}
\mathcal{S}_{q}(z)=\sum_{n=0}^{\infty}\frac{(-1)^{n}q^{n(n+1)}}{(q;q)_{2n+1}}z^{2n+1} \quad\mbox{ and }\quad
\mathcal{C}_{q}(z)=\sum_{n=0}^{\infty}\frac{(-1)^{n}q^{n^{2}}}{(q;q)_{2n}}z^{2n}, \label{eq:def_Sq_Cq}
\end{equation}
for $z\in\mathbb{C}$. With this choice, the $q$-version of Euler's identity
\begin{equation}
 \mathcal{E}_{q}(\mathrm{i}z)=\mathcal{C}_{q}(z)+\mathrm{i}q^{1/4}\mathcal{S}_{q}(z) \label{eq:qEuler_id}
\end{equation}
holds. Clearly,
\[
 \lim_{q\rightarrow1-}\mathcal{S}_{q}((1-q)z)=\sin z
 \quad \mbox{ and } \quad \lim_{q\rightarrow1-}\mathcal{C}_{q}((1-q)z)=\cos z.
\]
Alternatively, functions $\mathcal{S}_{q}$ and $\mathcal{C}_{q}$ can be written as the ${}_{1}\phi_{1}$ function with the base $q^{2}$,
 \begin{equation}
  \mathcal{S}_{q}(z)=\frac{z}{1-q}\,_{1}\phi_{1}(0;q^{3};q^{2},q^{2}z^{2}) \quad\mbox{and}\quad \mathcal{C}_{q}(z)=\,_{1}\phi_{1}(0;q;q^{2},qz^{2})
 \label{eq:SqCq_1phi1}
 \end{equation}
and in this form, they were introduced by Koornwinder and Swarttouw in \cite[Sec.~5]{KoornwinderSwarttouw}.
The formulas in \eqref{eq:SqCq_1phi1} allow to express the functions $\mathcal{S}_{q}$ and $\mathcal{C}_{q}$ in terms of the third Jackson (or Hahn-Exton or $_{1}\phi_{1}$) $q$-Bessel function
\[
 J_{\nu}^{(3)}(z;q)=\frac{(q^{\nu+1};q){}_{\infty}}{(q;q)_{\infty}}\, z^{\nu}\,_{1}\phi_{1}(0;q^{\nu+1};q,qz^{2}),
\]
see, for example, \cite{KoelinkSwarttouw, KoornwinderSwarttouw}. Namely, one has
\begin{equation}
 \mathcal{S}_{q}(z)=\frac{(q^{2};q^{2})_{\infty}}{(q;q^{2})_{\infty}}z^{1/2}J_{1/2}^{(3)}(z;q^{2}) \quad \mbox{ and } \quad 
 \mathcal{C}_{q}(z)=\frac{(q^{2};q^{2})_{\infty}}{(q;q^{2})_{\infty}}z^{1/2}J_{-1/2}^{(3)}(q^{-1/2}z;q^{2}) \label{eq:SqCq_rel_H-E-Bessel}
\end{equation}
which, if compared with the classical formulas
\[
 \sin z=\sqrt{\frac{\pi z}{2}}J_{1/2}(z) \quad \mbox{ and } \quad \cos z=\sqrt{\frac{\pi z}{2}}J_{-1/2}(z),
\]
gives another reason to refer to $\mathcal{S}_{q}$ and $\mathcal{C}_{q}$ as $q$-sine and $q$-cosine, respectively.
For some additional properties concerning $\mathcal{S}_{q}$ and $\mathcal{C}_{q}$, consult also paper of Bustoz and Cardoso
\cite{BustozCardoso}.

Recall that the functions $\mathcal{S}_{q}$ and $\mathcal{C}_{q}$ have infinite number of real interlacing zeros which are all simple.
This may be deduced from results of \cite[Sec.~3]{KoelinkSwarttouw} for $J_{\nu}^{(3)}$ applying the relations (\ref{eq:SqCq_rel_H-E-Bessel}).
Further, it is straightforward to verify that
 \begin{equation}
  \mathcal{S}_{q}(z)-\mathcal{S}_{q}(qz)=z\mathcal{C}_{q}(q^{1/2}z) \quad \mbox{ and } \quad \mathcal{C}_{q}(z)-\mathcal{C}_{q}(qz)=-zq^{1/2}\mathcal{S}_{q}(q^{1/2}z). \label{eq:qdif_SqCq}
 \end{equation}
Consequently, functions $\mathcal{S}_{q}$ and $\mathcal{C}_{q}$ are linearly independent solutions to the second-order $q$-difference equation
\[
u(q^{2}z)+\left(q^{2}z^{2}-(1+q)\right)u(qz)+qu(z)=0.
\]
It is also remarkable that there is a symmetry while interchanging $q$ and $q^{-1}$, namely
 \[
  \mathcal{E}_{q^{-1}}(q^{-1/4}z)=\mathcal{E}_{q}(-q^{1/4}z)
 \]
 and
 \[
   q^{-1/4}\mathcal{S}_{q^{-1}}(q^{-1/4}z)=-q^{1/4}\mathcal{S}_{q}(q^{1/4}z) \quad \mbox{ and } \quad \mathcal{C}_{q^{-1}}(q^{-1/4}z)=\mathcal{C}_{q}(q^{1/4}z).
 \]
 
At last, let us define the corresponding $q$-analogue to the hyperbolic sine and cosine by formulas
\begin{equation}
 \mathcal{S}h_{q}(z)=-\mathrm{i}\mathcal{S}_{q}(\mathrm{i}z)  \quad \mbox{ and } \quad \mathcal{C}h_{q}(z)=\mathcal{C}_{q}(\mathrm{i}z), \label{eq:def_ShqChq}
\end{equation}
for all $z\in\mathbb{C}$.

Finally, we recall an addition formula for $\mathcal{E}_{q}$. It has been originally derived by Rahman \cite{Rahman} and reproved by Suslov, 
see \cite[Thm.~3.6]{Suslov-book}. Since several important results stated below are a consequence of this formula we provide an independent
simple proof.


\begin{prop}For $u,v\in\mathbb{C}$, it holds
 \begin{eqnarray}
  & &
  \mathcal{E}_{q}(u)\mathcal{E}_{q}(-v)=
 {}_{3}\phi_{3}\left[\begin{matrix}
                       0, & u^{-1}vq^{1/2}, & uv^{-1}q^{1/2} \\
                       q^{1/2}, & -q^{1/2}, & -q
                      \end{matrix}\ ;q, uvq^{1/2} \right]\nonumber\\
  & &
  \hskip66pt+q^{1/4}\frac{u-v}{1-q}\,_{3}\phi_{3}\left[\begin{matrix}
                       0, & u^{-1}vq, & uv^{-1}q \\
                       q^{3/2}, & -q^{3/2}, & -q
                      \end{matrix}\ ;q, uvq \right]\!. \label{eq:E_q_u_times_E_q_v}
 \end{eqnarray}
\end{prop}

\begin{proof}
 By multiplication, we obtain
 \[
 \mathcal{E}_{q}(u)\mathcal{E}_{q}(-v)=
 \sum_{n=0}^{\infty}\frac{(-1)^{n}q^{n^{2}/4}}{(q;q)_{n}}v^{n}\sum_{k=0}^{n}{n \brack k}_{q}q^{k(k-1)/2}\left(-q^{\frac{1-n}{2}}\frac{u}{v}\right)^{k}.
 \]
 The inner sum equals $(q^{(1-n)/2}uv^{-1};q)_{n},$ as it follows from the formula \cite[Exer.~1.2]{GasperRahman}
 \[
 \sum_{k=0}^{n}{n \brack k}_{q}q^{k(k-1)/2}(-z)^{k}=(z;q)_{n}\,.
 \]
Now, it suffices to write separately the series with odd and even summation index $n$ and take into account that
 \[
  (q^{(1-n)/2}uv^{-1};q)_{n}=\begin{cases}
                            (-1)^{k}u^{k}v^{-k}q^{-k^{2}/2}(u^{-1}vq^{1/2};q)_{k}(uv^{-1}q^{1/2};q)_{k}, & \mbox {if } n=2k,\\
                            (-1)^{k}u^{k}v^{-k}q^{-k(k+1)/2}(u^{-1}vq;q)_{k}(uv^{-1};q)_{k+1}, & \mbox {if } n=2k+1.
                           \end{cases}
 \]
In this way, one arrives at the expression
 \begin{eqnarray*}
  & &
  \mathcal{E}_{q}(u)\mathcal{E}_{q}(-v)=
  \sum_{n=0}^{\infty}\frac{(-1)^{n}q^{n^{2}/2}}{(q;q)_{2n}}(u^{-1}vq^{1/2};q)_{n}(uv^{-1}q^{1/2};q)_{n}\ u^{n}v^{n}\\
  & &
  \hskip66pt +q^{1/4}(u-v)\sum_{n=0}^{\infty}\frac{(-1)^{n}q^{n(n+1)/2}}{(q;q)_{2n+1}}(u^{-1}vq;q)_{n}(uv^{-1}q;q)_{n}\ u^{n}v^{n}.
 \end{eqnarray*}
\end{proof}

As an immediate consequence of (\ref{eq:E_q_u_times_E_q_v}) and (\ref{eq:qEuler_id}), one has the following statement, cf. also \cite[Thm.~3.8]{Suslov-book}.

\begin{cor}
For $u,v\in\mathbb{C}$, one has
 \begin{equation}
 \mathcal{C}_{q}(u)\mathcal{C}_{q}(v)+q^{1/2}\mathcal{S}_{q}(u)\mathcal{S}_{q}(v)
 ={}_{3}\phi_{3}\left[\begin{matrix}
                       0, & u^{-1}vq^{1/2}, & uv^{-1}q^{1/2} \\
                       q^{1/2}, & -q^{1/2}, & -q
                      \end{matrix}\ ;q, -uvq^{1/2} \right] \label{eq:Cq_u_Cq_v_plus_Sq_u_Sq_v}
 \end{equation}
 and
 \begin{equation}
 \mathcal{S}_{q}(u)\mathcal{C}_{q}(v)-\mathcal{C}_{q}(u)\mathcal{S}_{q}(v)
 =\frac{u-v}{1-q}\,_{3}\phi_{3}\left[\begin{matrix}
                       0, & u^{-1}vq, & uv^{-1}q \\
                       q^{3/2}, & -q^{3/2}, & -q
                      \end{matrix}\ ;q, -uvq \right]\!. \label{eq:Cq_u_Sq_v_minus_Sq_u_Cq_v}
 \end{equation}
\end{cor}

Note that by setting $u=q^{1/2}v$ in (\ref{eq:Cq_u_Cq_v_plus_Sq_u_Sq_v}) one gets
\begin{equation}
 \mathcal{C}_{q}(q^{1/2}v)\mathcal{C}_{q}(v)+q^{1/2}\mathcal{S}_{q}(q^{1/2}v)\mathcal{S}_{q}(v)=1. \label{eq:Cq_Cq+Sq_Sq_eq_1}
\end{equation}
In addition, from (\ref{eq:qEuler_id}) and (\ref{eq:Cq_u_Cq_v_plus_Sq_u_Sq_v}), it readily follows
\begin{equation}
 |\mathcal{E}_{q}(ix)|^{2}=\mathcal{C}_{q}^{2}(x)+q^{1/2}\mathcal{S}_{q}^{2}(x)={}_{2}\phi_{2}\left[\begin{matrix}
                       0, & q^{1/2} \\
                       -q^{1/2}, & -q
                      \end{matrix}\ ;q, -x^{2}q^{1/2} \right], \quad \mbox{for } x\in\mathbb{R}. \label{eq:abs_val_Eq}
\end{equation}

\section{Relations between $\mathcal{E}_{q}$ and $q^{-1}$-Fibonacci polynomials}\label{sec:rel_qE_eFib}

Recall that rather than assuming $q>1$, it is more convenient to write $q^{-1}$ instead of $q$ in (\ref{eq:def_T_n}) and still suppose $0<q<1$. 
Then one arrives at the $q^{-1}$-Fibonacci polynomials with the explicit expression:
\begin{equation}
 \varphi_{n}(x;q^{-1})=\sum_{2k<n}{n-k-1\brack k}_{q}q^{k(k+1-n)}x^{n-2k-1}, \quad \mbox{ for } n\in\mathbb{Z}_{+}. \label{eq:def_varphi_n_recip_q}
\end{equation}
Note the sequence 
\[
\left\{q^{(n-1)^{2}/2}\varphi_{n}\left(q^{-(n-1)/2}x;q^{-1}\right)\right\}_{n=0}^{\infty}
\]
satisfies the same recurrence (\ref{eq:q-Fib_recur}) and initial conditions as $\{\varphi_{n}(x;q)\}_{n=0}^{\infty}$
Hence, we have the following relation between $q$-Fibonacci and $q^{-1}$-Fibonacci polynomials,
\[
 \varphi_{n}(x;q)=q^{(n-1)^{2}/2}\varphi_{n}\left(q^{-(n-1)/2}x;q^{-1}\right), \quad \forall n\in\mathbb{Z}_{+},
\]
cf. \cite[Eq.~(2.6)]{Carlitz75}.

First, we derive an expression for the $q^{-1}$-Fibonacci polynomials $\varphi_{n}(x;q^{-1})$ in terms of the $q$-exponential function $\mathcal{E}_{q}$.
A limit relation for $\varphi_{n}(x;q^{-1})$, as $n\rightarrow\infty$, then readily follows.

\begin{prop}\label{prop:recip_qFib_eq_wronsk_qExp}
 For all $n\in\mathbb{Z}_{+}$, one has
 \[
   \varphi_{n}(x;q^{-1})=\frac{1}{2}q^{-(n-1)^{2}/4}\left(\mathcal{E}_{q}(x)\mathcal{E}_{q}(-q^{n/2}x)-(-1)^{n}\mathcal{E}_{q}(-x)\mathcal{E}_{q}(q^{n/2}x)\right).
 \]
\end{prop}

\begin{proof}
 For $n\in\mathbb{Z}_{+}$ and $x\in\mathbb{C}$, put
 \[
  \chi_{n}^{\pm}=(\mp 1)^{n}q^{-n(n-2)/4}\mathcal{E}_{q}(\pm q^{n/2}x).
 \]
 By using (\ref{eq:E_q_recur}), one readily verifies that the sequences $\{\chi_{n}^{+}\}$ and $\{\chi_{n}^{-}\}$ fulfill the recurrence (\ref{eq:q-Fib_recur}) with $q$ replaced by $q^{-1}$.
 Next, with the aid of the identity (\ref{eq:E_q_u_times_E_q_v}), one obtains
 \[
  \chi_{0}^{+}\chi_{1}^{-}-\chi_{0}^{-}\chi_{1}^{+}=q^{1/4}\left(\mathcal{E}_{q}(x)\mathcal{E}_{q}(-q^{1/2}x)+\mathcal{E}_{q}(-x)\mathcal{E}_{q}(q^{1/2}x)\right)=2q^{1/4}.
 \]
 Finally, the sequence 
 \[
 \left\{\frac{\chi_{0}^{+}\chi_{n}^{-}-\chi_{0}^{-}\chi_{n}^{+}}{\chi_{0}^{+}\chi_{1}^{-}-\chi_{0}^{-}\chi_{1}^{+}}\right\}_{n=0}^{\infty}
 \]
 coincides with $\{\varphi_{n}(x;q^{-1})\}_{n=0}^{\infty}$ since it satisfies the same recurrence with the same initial conditions. The statement now follows. 
\end{proof}

The following corollary readily follows from  (\ref{eq:qEuler_id}), (\ref{eq:def_ShqChq}), and Proposition \ref{prop:recip_qFib_eq_wronsk_qExp}.

\begin{cor}
   For $n\in\mathbb{Z}_{+}$, one has
  \begin{equation}
   \varphi_{2n+1}(x;q^{-1})=q^{-n^{2}}\left(\mathcal{C}h_{q}(x)\mathcal{C}h_{q}(q^{n+1/2}x)-q^{1/2}\mathcal{S}h_{q}(x)\mathcal{S}h_{q}(q^{n+1/2}x)\right) \label{eq:qFib_2n_eq_ShqChq}
  \end{equation}
 and
  \begin{equation}
   \varphi_{2n}(x;q^{-1})=q^{-n(n-1)}\left(\mathcal{S}h_{q}(x)\mathcal{C}h_{q}(q^{n}x)-\mathcal{C}h_{q}(x)\mathcal{S}h_{q}(q^{n}x)\right). \label{eq:qFib_2n+1_eq_ShqChq}
  \end{equation}
\end{cor}

In addition, the identities (\ref{eq:qFib_2n_eq_ShqChq}), (\ref{eq:qFib_2n+1_eq_ShqChq}) together with the formulas (\ref{eq:Cq_u_Cq_v_plus_Sq_u_Sq_v}), (\ref{eq:Cq_u_Sq_v_minus_Sq_u_Cq_v}),
and (\ref{eq:def_ShqChq}) allow to express $q^{-1}$-Fibonacci polynomials as terminating ${}_{3}\phi_{3}$ series. These identities can be also readily verified by using \eqref{eq:def_varphi_n}.

\begin{cor}
 For $n\in\mathbb{Z}_{+}$, one has
 \[
  \varphi_{2n+1}(x;q^{-1})=q^{-n^{2}}\,_{3}\phi_{3}\left[\begin{matrix}
							  0, & q^{-n}, & q^{n+1} \\
							  q^{1/2}, & -q^{1/2}, & -q
							  \end{matrix}\ ;q,q^{n+1}x^{2}\right]
 \]
 and
 \[  
  \varphi_{2n}(x;q^{-1})=q^{-n(n-1)}\frac{1-q^{n}}{1-q}x\,_{3}\phi_{3}\left[\begin{matrix}
							  0, & q^{-n+1}, & q^{n+1} \\
							  q^{3/2}, & -q^{3/2}, & -q
							  \end{matrix}\ ;q,q^{n+1}x^{2}\right]\!.
 \]
\end{cor}

\begin{prop}\label{prop:qFib_limits}
For all $x\in\mathbb{C}$, the limit relations 
 \[
  \lim_{n\rightarrow\infty}q^{n(n-1)}\varphi_{2n}(x;q^{-1})=\mathcal{S}h_{q}(x) \quad\mbox{ and }\quad
  \lim_{n\rightarrow\infty}q^{n^{2}}\varphi_{2n+1}(x;q^{-1})=\mathcal{C}h_{q}(x)
 \]
hold.
\end{prop}

\begin{proof}
 Send $n\to\infty$ in identities (\ref{eq:qFib_2n_eq_ShqChq}) and (\ref{eq:qFib_2n+1_eq_ShqChq}).
\end{proof}

The last formulas relating $q^{-1}$-Fibonacci polynomials with the functions $\mathcal{E}_{q}$, $\mathcal{S}h_{q}$, and $\mathcal{C}h_{q}$ follows from an identity which may be viewed as a $q$-analogue
of the famous Euler-Cassini's identity for Fibonacci numbers. We present a general way of proving such identities by using a function called 
$\mathfrak{F}$ that was originally introduced in \cite{StampachStovicek11}. This function takes as its argument a complex sequence satisfying certain convergence condition.
Here, it suffices to recall the definition of $\mathfrak{F}$ applied to any $n$-tuple of complex variables $(x_{1},x_{2},\dots,x_{n})$. This definition can be formulated as follows:
\[
 \mathfrak{F}(x_{1},x_{2},\dots,x_{n})=\det X^{(n)}, \mbox{ for } n\in\mathbb{N},\mbox{ and } \mathfrak{F}(\emptyset)=1,
\]
where $X^{(n)}$ is the $n\times n$ matrix with entries
$$X_{ij}^{(n)}=\begin{cases}
                1, & \mbox{ if } i=j,\\
                x_{i}, & \mbox{ if } |i-j|=1,\\
                0, & \mbox{ otherwise}.
               \end{cases}$$

The function $\mathfrak{F}$ satisfies the relation 
\begin{equation}
\mathfrak{F}(x_{1},x_{2},\dots,x_{k+1})=\mathfrak{F}(x_{1},x_{2},\dots,x_{k})-x_{k}x_{k+1}\,\mathfrak{F}(x_{1},x_{2},\dots,x_{k-1}), \quad k\in\mathbb{N}\label{eq:mathfrakF_reccur}
\end{equation}
and the identity
 \begin{eqnarray}
 &  & \mathfrak{F}(x_{1},x_{2},\dots,x_{n})\mathfrak{F}(x_{2},x_{3},\dots,x_{m})-\mathfrak{F}(x_{1},x_{2},\dots,x_{m})\mathfrak{F}(x_{2},x_{3},\dots,x_{n})\nonumber\\
 &  & \hskip160pt=\left(\prod_{j=1}^{n}x_{j}x_{j+1}\right)\mathfrak{F}(x_{n+2},x_{n+3},\dots,x_{m})\label{eq:mathfrakF_rel_fund}
\end{eqnarray}
which holds true for $m,n\in\mathbb{Z}_{+}$, $n<m$, see \cite[Lemma 1]{StampachStovicekLAA13b}. 

By using (\ref{eq:mathfrakF_reccur}), one easily verifies that the sequence $\{p_{n}\}_{n=0}^{\infty}$ determined by the second-order recurrence
$ p_{n+1}=\alpha_{n}p_{n}+\beta_{n}p_{n-1}$, for $n\in\mathbb{Z}_{+}$, and initial conditions $p_{-1}=0$ and $p_{0}=1$, can be expressed with the aid of the function $\mathfrak{F}$, 
see, for example, the general formula \cite[Eq.~(19)]{StampachStovicek_Coulomb}. For $q^{-1}$-Fibonacci polynomials the formula reads
\begin{equation}
\varphi_{n+1}(x;q^{-1})=x^{n}\,\mathfrak{F}\left(\left\{{\rm i}x^{-1}q^{-(2k-1)/4}\right\}_{k=1}^{n}\right), \quad n\in\mathbb{Z}_{+}.
\label{eq:qFib_rel_mathfrakF}\end{equation}

Further we derive the well known $q$-Euler-Cassini formula for $\varphi_{n}(x;q^{-1})$ first proved by MacMahon \cite[Chp.~III, Sec.~VII]{MacMahonII} 
and reproved several times. See, for example, \cite{BerkovichPaule, Cigler04} for combinatorial proofs. Another proof 
for the more general polynomials introduced by Al-Salam and Ismail in \cite{Al-SalamIsmail} was given in \cite{Ismail_etal} (the presented method can be applied to this
more general case as well). For related investigations consult also \cite{Andrews_etal} and references therein.

\begin{prop}\label{prop:Euler-Cassini}
 For $m,n\in\mathbb{Z}_{+}$, $n\leq m$, it holds
 \begin{eqnarray}
  & &
  q^{\frac{n}{2}}\varphi_{m}(x;q^{-1})\varphi_{n+1}(q^{-\frac{1}{2}}x;q^{-1})-q^{\frac{m}{2}}\varphi_{n}(x;q^{-1})\varphi_{m+1}(q^{-\frac{1}{2}}x;q^{-1})\nonumber\\
  & &
  \hskip200pt =(-1)^{n}q^{n-\frac{mn}{2}}\varphi_{m-n}(q^{\frac{n}{2}}x;q^{-1}) \label{eq:Euler-Cassini}.
 \end{eqnarray}
\end{prop}

\begin{proof}
By substituting 
\[
 x_{k}={\rm i}x^{-1}q^{-(2k-1)/4}
\]
in (\ref{eq:mathfrakF_rel_fund}) and using identity (\ref{eq:qFib_rel_mathfrakF}) one arrives at the relation (\ref{eq:Euler-Cassini})
for $n<m$. Clearly, in the case $n=m$, both sides of (\ref{eq:Euler-Cassini}) vanish.
\end{proof}

\begin{prop}
 For all $n\in\mathbb{Z}_{+}$ and $x\in\mathbb{C}$, the following relations hold true:
 \begin{eqnarray*}
  \hskip-8pt&i)&\varphi_{2n+1}(q^{-1/2}x;q^{-1})\mathcal{S}h_{q}(x)-q^{-n}\varphi_{2n}(x;q^{-1})\mathcal{C}h_{q}(q^{-1/2}x)=q^{-n^{2}}\mathcal{S}h_{q}(q^{n}x),\\
  \hskip-8pt&ii)&\varphi_{2n+1}(q^{-1/2}x;q^{-1})\mathcal{C}h_{q}(x)-q^{-n+1/2}\varphi_{2n}(x;q^{-1})\mathcal{S}h_{q}(q^{-1/2}x)=q^{-n^{2}}\mathcal{C}h_{q}(q^{n}x),\\
  \hskip-8pt&iii)&\varphi_{2n+1}(x;q^{-1})\mathcal{S}h_{q}(q^{-1/2}x)-q^{n}\varphi_{2n+2}(q^{-1/2}x;q^{-1})\mathcal{C}h_{q}(x)=q^{-n^{2}}\mathcal{S}h_{q}(q^{n+1/2}x),\\
  \hskip-8pt&iv)&\varphi_{2n+1}(x;q^{-1})\mathcal{C}h_{q}(q^{-1/2}x)-q^{n+1/2}\varphi_{2n+2}(q^{-1/2}x;q^{-1})\mathcal{S}h_{q}(x)=q^{-n^{2}}\mathcal{C}h_{q}(q^{n+1/2}x),\\
  \hskip-8pt&v)&\varphi_{n+1}(q^{-1/2}x;q^{-1})\mathcal{E}_{q}(x)-q^{-(2n-1)/4}\varphi_{n}(x;q^{-1})\mathcal{E}_{q}(q^{-1/2}x)=(-1)^{n}q^{-n^{2}/4}\mathcal{E}_{q}(q^{n/2}x).
 \end{eqnarray*}
\end{prop}

\begin{proof}
 This statement follows from Propositions \ref{prop:qFib_limits} and \ref{prop:Euler-Cassini}. First, we indicate the derivation of identity $i)$. Relations $ii)$, $iii)$, and $iv)$ can be derived in a similar way.
 In the second part of the proof, we verify equality $v)$.

 By writing $2m$ instead of $m$ and $2n$ instead of $n$ in (\ref{eq:Euler-Cassini}), and multiplying both sides by $q^{m(m-1)-n}$, one finds
 \begin{eqnarray*}
  & &
  \varphi_{2n+1}(q^{-1/2}x;q^{-1})\left[q^{m(m-1)}\varphi_{2m}(x;q^{-1})\right]
  -q^{-n}\varphi_{2n}(x;q)\left[q^{m^{2}}\varphi_{2m+1}(q^{-1/2}x;q^{-1})\right]\\
  & & \hskip200pt =q^{-n^{2}}\left[q^{(m-n)(m-n-1)}\varphi_{2(m-n)}(q^{n}x;q^{-1})\right]\!.
 \end{eqnarray*}
 To obtain identity $i)$, it is sufficient to send $m\rightarrow\infty$ and apply Proposition \ref{prop:qFib_limits} to the expressions in square brackets.

  Next, recall that by (\ref{eq:qEuler_id}) and (\ref{eq:def_ShqChq}), one has
  \begin{equation}
  \mathcal{E}_{q}(z)=\mathcal{C}h_{q}(z)+q^{1/4}\mathcal{S}h_{q}(z), \label{eq:qEuler_hyperbol}
  \end{equation}
  for all $z\in\mathbb{C}$. By multiplying the equation $i)$ by $q^{1/4}$ and adding to the equation $ii)$, one obtains
  \[
  \varphi_{2n+1}(q^{-1/2}x;q^{-1})\mathcal{E}_{q}(x)-q^{-n+1/4}\varphi_{2n}(x;q^{-1})\mathcal{E}_{q}(q^{-1/2}x)=q^{-n^{2}}\mathcal{E}_{q}(q^{n}x)
  \]
  where the formula (\ref{eq:qEuler_hyperbol}) has been used. This verifies $v)$ for an even index. 
  Similarly, by multiplying the equation $iii)$ by $q^{1/4}$ and adding to the equation $iv)$, one obtains
  \[
  \varphi_{2n+1}(x;q^{-1})\mathcal{E}_{q}(q^{-1/2}x)-q^{n+1/4}\varphi_{2n+2}(q^{-1/2}x;q^{-1})\mathcal{E}_{q}(x)=q^{-n^{2}}\mathcal{E}_{q}(q^{n+1/2}x).
  \]
  which is the identity $v)$ with an odd index.
\end{proof}

\section{Related orthogonal polynomials, reproducing kernel, and N-extremal measures of orthogonality} \label{sec:OPs_repro_ker_Next_meas}

In this section, we draw our attention to a family of orthogonal polynomials associated to
$q^{-1}$-Fibonacci polynomials. The main goal is the description of all their measures of orthogonality
which are N-extremal solutions of the corresponding Hamburger moment problem. Main references concerning
the Hamburger moment problem are classical treatises \cite{Akhiezer} and \cite{ShohatTamarkin}, for a brief 
summary, see also \cite[Sec.~21.1]{Ismail-book}.

Recall that polynomials $P_{n}$ and $Q_{n}$, $n\in\mathbb{Z}_{+}$, are orthogonal polynomials 
of the first and second kind, respectively, if they are solutions of the recurrence relation
\[
 \beta_{n-1}u_{n-1}+(\alpha_{n}-z)u_{n}+\beta_{n}u_{n+1}=0, \quad n\in\mathbb{N},
\]
satisfying initial conditions $P_{0}(z)=1$, $P_{1}(z)=(z-\alpha_{0})/\beta_{0}$, $Q_{0}(z)=0$ and $Q_{1}(z)=1/\beta_{0}$,
see \cite[Chp.~1]{Akhiezer}. It is to be assumed that $\alpha_{n}\in\mathbb{R}$ and $\beta_{n}\in\mathbb{R}\setminus\{0\}$, 
for all $n\in\mathbb{Z}_{+}$.

For $n=-1,0,1,2,\dots$, define
\begin{equation}
 P_{n}(x;q)=(-{\rm i})^{n}q^{n(n+1)/4}\varphi_{n+1}({\rm i}q^{-1/2}x;q^{-1}). \label{eq:def_P_n}
\end{equation}
The polynomials $P_{n}(x;q)$ form the family of orthogonal polynomials of the first kind 
with $\alpha_{n}=0$ and $\beta_{n}=q^{-n/2}$, $n\in\mathbb{Z}_{+}$. In addition, they are
the symmetric counterpart to the polynomials $T_{n}(x;q^{-1})$ since their monic version coincides with 
$T_{n}(x;q^{-1})$, namely
\[
 q^{-n(n-1)/4}P_{n}(x;q)=T_{n}(x;q^{-1}).
\]
For the corresponding orthogonal polynomials of the second kind $Q_{n}(x;q)$, one has
\begin{equation}
 Q_{n}(x;q)=P_{n-1}(q^{1/2}x;q)=(-{\rm i})^{n-1}q^{n(n-1)/4}\varphi_{n}({\rm i}x;q^{-1}), \quad n\in\mathbb{Z}_{+}. \label{eq:rel_OPs_first_second}
\end{equation}

Taking into account the equation (\ref{eq:def_varphi_n_recip_q}), one easily verifies that
\begin{equation}
 P_{2n-1}(0;q)=Q_{2n}(0;q)=0 \;\mbox{ and }\; P_{2n}(0;q)=Q_{2n+1}(0;q)=(-1)^{n}q^{n/2} \label{eq:OPs_at_origin}
\end{equation}
where $n\in\mathbb{Z}_{+}$. Recall that the Hamburger moment problem associated with a family of orthogonal polynomials 
of the first kind $P_{n}$ (and of the second kind $Q_{n}$) is in the indeterminate case if and only if
\[
\sum_{n=0}^{\infty}\left(P_{n}^{2}(0)+Q_{n}^{2}(0)\right)<\infty,
\]
see, for example, \cite[Chp.~2]{Akhiezer}. For polynomials $P_{n}(0;q)$ and $Q_{n}(0;q)$, this condition is clearly 
true for all $q$ such that $0<q<1$, as it readily follows from (\ref{eq:OPs_at_origin}).

Further recall that, by the Nevanlinna theorem, all measures of orthogonality $\mu_{\varphi}$ of polynomials $P_{n}$ 
are in one-to-one correspondence with functions $\varphi$ belonging to the one-point compactification 
of the space of Pick functions. Recall that Pick functions are defined and holomorphic on the open complex halfplane $\Im z>0$,
with values in the closed halfplane $\Im z\geq0$.
The correspondence is established by identifying the Stieltjes transform of the measure $\mu_{\varphi}$,
\[
\int_{\mathbb{R}}\frac{\mbox{d}\mu_{\varphi}(x)}{z-x}=\frac{A(z)\varphi(z)-C(z)}{B(z)\varphi(z)-D(z)}\,,\ \ z\in\mathbb{C}\setminus\mathbb{R}.
\]
Entire functions $A$, $B$, $C$, $D$ are local uniform limits of polynomials $A_{n}$, $B_{n}$, $C_{n}$, $D_{n}$, respectively, where
\begin{eqnarray}
  & &
  A_{n+1}(z)=\beta_{n}\left(Q_{n+1}(z)Q_{n}(0)-Q_{n+1}(0)Q_{n}(z)\right),\label{eq:A_n+1}\\
  & &
  B_{n+1}(z)=\beta_{n}\left(P_{n+1}(z)Q_{n}(0)-Q_{n+1}(0)P_{n}(z)\right),\label{eq:B_n+1}\\
  & &
  C_{n+1}(z)=\beta_{n}\left(Q_{n+1}(z)P_{n}(0)-P_{n+1}(0)Q_{n}(z)\right),\label{eq:C_n+1}\\
  & &
  D_{n+1}(z)=\beta_{n}\left(P_{n+1}(z)P_{n}(0)-P_{n+1}(0)P_{n}(z)\right),\label{eq:D_n+1}
\end{eqnarray}
see \cite[Sec.~21.1]{Ismail-book}.

The particular class of measures of orthogonality is composed by measures $\mu_{t}$ associated with the Pick function 
$\varphi(z)=t\in\mathbb{R}\cup\{\infty\}$. Measures $\mu_{t}$ are called N-extremal and are purely discrete. The support 
of $\mu_{t}$ is an unbounded set of isolated points which is known to be equal to the zero set
\begin{equation}
\mathfrak{Z}_{t}=\{x\in\mathbb{R} \mid B(x)t-D(x)=0\}.\label{eq:Zero_t}
\end{equation}
Hence
\begin{equation}
\mu_{t}=\sum_{x\in\mathfrak{Z}_{t}}\rho(x)\delta_{x}\label{eq:mu_t}
\end{equation}
where 
\begin{equation}
\rho(x)=\Res_{z=x}\frac{A(z)t-C(z)}{B(z)t-D(z)}=\frac{A(x)t-C(x)}{B'(x)t-D'(x)}=\frac{1}{B'(x)D(x)-B(x)D'(x)}\label{eq:rho}
\end{equation}
since, for $x\in\mathfrak{Z}_{t}$, $t=D(x)/B(x)$ and the general identity
\begin{equation}
A(z)D(z)-B(z)C(z)=1, \label{eq:AD-BC=1}
\end{equation}
holds true for all $z\in\mathbb{C}$. The orthogonality relation for polynomials of the first kind then reads
\begin{equation}
 \sum_{x\in\mathfrak{Z}_{t}}\rho(x)P_{n}(x)P_{m}(x)=\delta_{mn}, \quad \forall m,n\in\mathbb{Z}_{+}. \label{eq:og_rel_Nextremal_general}
\end{equation}

Thus, the determination of the Nevanlinna functions, $B$ and $D$ in particular, is essential for the description
of N-extremal measures of orthogonality. Formulas for $A$, $B$, $C$ and $D$ can be, in turn, established if the asymptotic
behavior of $P_{n}$ and $Q_{n}$ is known for $n\rightarrow\infty$, as one observes from formulas (\ref{eq:A_n+1})-(\ref{eq:D_n+1}).
In the case of $P_{n}(x;q)$ and $Q_{n}(x;q)$, the asymptotic formulas follow from Proposition \ref{prop:qFib_limits}
and the relations (\ref{eq:def_P_n}) and (\ref{eq:rel_OPs_first_second}). We arrive at the following limits:
 \begin{equation}
  \lim_{n\rightarrow\infty} (-1)^{n}q^{-n/2}P_{2n}(x;q)=\mathcal{C}_{q}(q^{-1/2}x) \label{eq:P_n_limit_even}
 \end{equation}
and
 \begin{equation}
 \lim_{n\rightarrow\infty} (-1)^{n}q^{-n/2}P_{2n+1}(x;q)=q^{1/2}\mathcal{S}_{q}(q^{-1/2}x). \label{eq:P_n_limit_odd}
 \end{equation}
 
An alternative way of derivation of the last two limit relations was presented
by Chen and Ismail \cite{ChenIsmail}. It is based on Darboux's method applied to a suitable generating function
formula for $P_{n}(x;q)$. According to this method, the leading asymptotic term of $P_{n}(x;q)$ is determined by the 
singularity of the generating function which is located most closely to the origin, cf. \cite[Section~8.9]{Olver}.
The explicit form of the generating function for $P_{n}(x;q)$ is as follows \cite[Thm.~3.1]{ChenIsmail}
\[
    \sum_{n=0}^{\infty}P_{n}(x;q)t^{n}=\sum_{k=0}^{\infty}\frac{q^{k(k-1)/4}x^{k}t^{k}}{(-q^{1/2}t^{2};q)_{k+1}}, \quad |t|<q^{-1/4}.
\]

Let us note there is another (entire) generating function formula which is formulated in the following proposition
in terms of $q^{-1}$-Fibonacci polynomials.

\begin{prop} For all $x,s\in\mathbb{C}$, it holds
 \begin{equation}
             \sum_{n=0}^{\infty}q^{n(n-1)/2}\varphi_{n+1}(x;q^{-1})(-s)^{n}={}_{2}\phi_{2}\left[\begin{matrix}
                       q, & x^{-1}s \\
                       0, & 0
                      \end{matrix}\ ;q, sx \right]\!.\label{eq:gener_func_analytic}
            \end{equation}
\end{prop}

\begin{proof}
Suppose $x\in\mathbb{C}$ is fixed and denote by $V(s)$ the function on the right side of (\ref{eq:gener_func_analytic}).
$V(s)$ is a well defined entire function which is readily seen to satisfy the $q$-difference equation
\begin{equation}
V(s)=1-s(x-s)V(qs).\label{eq:q-diff_V_first}
\end{equation}
Writing the power series expansion of $V(s)$ at $s=0$ in the form
\[
V(s)=\sum_{n=0}^{\infty}v_{n}q^{n(n-1)/2}(-s)^{n}
\]
and inserting the series into (\ref{eq:q-diff_V_first}) one finds that
the coefficients $v_{n}$ obey the recurrence $v_{n+1}=xv_{n}+q^{-n}v_{n-1}$, for $n\in\mathbb{N}$,
with the initial conditions $v_{0}=1$ and $v_{1}=x$. Necessarily, $v_{n}=\varphi_{n+1}(x;q^{-1})$ for all $n\in\mathbb{Z}_{+}$.
\end{proof}

By applying the limit $n\to\infty$ in the equations (\ref{eq:A_n+1})-(\ref{eq:D_n+1}), using (\ref{eq:P_n_limit_even}), 
(\ref{eq:P_n_limit_odd}), and (\ref{eq:rel_OPs_first_second}), one obtains formulas for the Nevanlinna functions
\begin{equation}
 A(z;q)=q^{-1/2}D(q^{1/2}z;q)=\mathcal{S}_{q}(z) \quad\mbox{and}\quad
 C(z;q)=-B(q^{1/2}z;q)=\mathcal{C}_{q}(z).\label{eq:Nevan_ABCD}
\end{equation}

\begin{rem}
Note that in the particular case of the functions from (\ref{eq:Nevan_ABCD}) the general identity (\ref{eq:AD-BC=1}) becomes (\ref{eq:Cq_Cq+Sq_Sq_eq_1}).
\end{rem}

From the general theory, it is well known that for any $t\in\mathbb{R}$, the
functions $A(z)t+C(z)$ and $B(z)t+D(z)$ have infinite number of zeros, all of them being real and simple.
Moreover, zeros of $A(z)t+C(z)$ and $B(z)t+D(z)$ interlace. The same holds true for functions $A(z)$ and $B(z)$ 
which correspond to the case $t=\infty$, see \cite[Chp.~2]{Akhiezer}. From this and (\ref{eq:Nevan_ABCD})
one immediately obtains the following statement.

\begin{prop}\label{prop:zeros_real_simple_interlacing}
 All zeros of the functions $\mathcal{S}_{q}$ and $\mathcal{C}_{q}$ are real and simple. In addition, zeros
 of the functions $\mathcal{S}_{q}(z)$ and $\mathcal{C}_{q}(q^{-1/2}z)$ interlace, as well 
 as zeros of the functions $\mathcal{S}_{q}(q^{-1/2}z)$ and $\mathcal{C}_{q}(z)$ (as functions in $z$).
\end{prop}

As shown by Buchwalter and Cassier in \cite{BuchwalterCassier}, see also \cite[Sec.~1]{Berg95}, the reproducing kernel for orthogonal polynomials 
of the first kind can be expressed in terms of functions $B(z)$ and $D(z)$, namely
\begin{equation}
K(u,v)=\sum_{n=0}^{\infty}P_{n}(u)P_{n}(v)=\frac{B(u)D(v)-D(u)B(v)}{u-v}\,.\label{eq:repro_kernel}
\end{equation}
In the case of polynomials $P_{n}(x;q)$, it is possible to express the reproducing kernel in a closed form.

\begin{thm}
 The formula for the reproducing kernel for orthogonal polynomials $P_{n}(x;q)$ reads
 \begin{equation}
  K(u,v)=\frac{1}{1-q}\,_{3}\phi_{3}\left[\begin{matrix}
                       0, & u^{-1}vq, & uv^{-1}q\\
                       q^{3/2}, & -q^{3/2}, & -q
                      \end{matrix}\ ;q, -uv \right]\!. \label{eq:repro_kernel_3phi3}
 \end{equation}
\end{thm}

\begin{proof}
 The statement follows from the formulas (\ref{eq:repro_kernel}), (\ref{eq:Nevan_ABCD}), and (\ref{eq:Cq_u_Sq_v_minus_Sq_u_Cq_v}).
\end{proof}

\begin{cor}\label{cor:recip_rho_eq_3phi3}
 It holds
 \[
 B'(u;q)D(u;q)-B(u;q)D'(u;q)=
 \frac{1}{1-q}\,_{3}\phi_{3}\left[\begin{matrix}
                       0, & q, & q\\
                       q^{3/2}, & -q^{3/2}, & -q
                      \end{matrix}\ ;q, -u^{2} \right]\!,
 \]
 for all $u\in\mathbb{C}$.
\end{cor}

\begin{proof}
 It suffices to send $v\rightarrow u$ in the formula (\ref{eq:repro_kernel_3phi3}).
\end{proof}


Further, we focus on the description of all N-extremal measures of orthogonality for polynomials $P_{n}(x;q)$.
We use a different parametrization of N-extremal measures $\mu_{t}$ substituting for $t$ the quotient of $q$-cosine and $q$-sine functions of a new parameter~$u$.
This reparametrization together with the formula (\ref{eq:E_q_u_times_E_q_v}) allow to express all N-extremal measures in a compact form.
A similar idea was used by Ismail and Masson in the case of $q$-Hermite polynomials \cite{IsmailMasson}.

\begin{thm}
 If $0<q<1$ and $u\in\mathbb{R}$, then the orthogonality relation for $P_{n}(x;q)$ reads
 \[
  \sum_{k=1}^{\infty}\left(\,_{3}\phi_{3}\left[\begin{matrix}
                       0, & q, & q\\
                       q^{3/2}, & -q^{3/2}, & -q
                      \end{matrix}\ ;q, -\lambda_{k}^{2}(u) \right]\right)^{\!-1}\!P_{n}(\lambda_{k}(u);q)
                      P_{m}(\lambda_{k}(u);q)
  =\frac{1}{1-q}\delta_{mn}
 \]
 where $\lambda_{1}(u),\lambda_{2}(u),\lambda_{3}(u),\dots$ stand for zeros of the function (\ref{eq:3phi3_zeros_supp}).
 Moreover, the map 
 \[
  u\mapsto\mu_{u}(q)=(1-q)\sum_{k=1}^{\infty}\left(\,_{3}\phi_{3}\left[\begin{matrix}
                       0, & q, & q\\
                       q^{3/2}, & -q^{3/2}, & -q
                      \end{matrix}\ ;q, -\lambda_{k}^{2}(u) \right]\right)^{\!-1}\delta_{\lambda_{k}(u)}
 \]
 represents one-to-one mapping from $[0,s_{1}(q))$ onto the set of all N-extremal measures associated
 with polynomials $P_{n}(x;q)$ where $s_{1}(q)$ stands for the lowest positive zero of~$\mathcal{S}_{q}$.
\end{thm}

\begin{proof}
Function
\[
 u\mapsto\frac{\mathcal{C}_{q}(u)}{\mathcal{S}_{q}(u)}
\]
is bijection of the interval $[0,s_{1}(q))$ onto $\mathbb{R}\cup\{\infty\}$. Hence, by identifying
\begin{equation}
 t=\frac{\mathcal{C}_{q}(u)}{\mathcal{S}_{q}(u)} \label{eq:reparam_t_qTan}
\end{equation}
in (\ref{eq:mu_t}), the one-to-one mapping $u\mapsto\mu_{u}(q)$ from the interval $[0,s_{1}(q))$ onto the
set of all N-extremal measures of orthogonality for $P_{n}(x;q)$ is established. The support of $\mu_{u}(q)$
coincides with the set of zeros of the function $z\mapsto B(z;q)t-D(z;q)$ where $t$ is related to $u$ by (\ref{eq:reparam_t_qTan}).
By using (\ref{eq:Nevan_ABCD}), one finds $\mu_{u}(q)$ is supported by set of zeros of the function
\[
z\mapsto\mathcal{C}_{q}(q^{-1/2}z)\mathcal{C}_{q}(u)+q^{1/2}\mathcal{S}_{q}(q^{-1/2}z)\mathcal{S}_{q}(u)
\]
which, according to (\ref{eq:Cq_u_Cq_v_plus_Sq_u_Sq_v}), equals
\begin{equation}
z\mapsto{}_{3}\phi_{3}\left[\begin{matrix}
                       0, & u^{-1}z, & uz^{-1}q \\
                       q^{1/2}, & -q^{1/2}, & -q
                      \end{matrix}\ ;q, -uz \right]\!. \label{eq:3phi3_zeros_supp}
\end{equation}
In addition, the formula given in Corollary \ref{cor:recip_rho_eq_3phi3} coincides with the reciprocal of the function $\rho$
from (\ref{eq:rho}). Thus, we get the description of all the N-extremal measures $\mu_{u}(q)$ in terms of ${}_{3}\phi_{3}$ in the form stated
as the second part of the statement. The ortho\-gonality relation then follows from (\ref{eq:og_rel_Nextremal_general}).
\end{proof}

\begin{rem}
The general theory on the Nevanlinna parametrization provides us with the following generalization of Proposition \ref{prop:zeros_real_simple_interlacing}:
\emph{For $u\in\mathbb{R}$, all the zeros $\lambda_{1}(u),\lambda_{2}(u),\dots$ of function (\ref{eq:3phi3_zeros_supp})
are real and simple. In addition, if we denote the function from (\ref{eq:3phi3_zeros_supp}) by $f_{u}(\,\cdot\,;q)$, then, for any $u_1,u_2\in[0,s_{1}(q))$ such that $u_{1}\neq u_{2}$, 
zeros of functions $f_{u_1}(\,\cdot\,;q)$ and $f_{u_2}(\,\cdot\,;q)$ interlace.}
\end{rem}

Let us draw our attention to two special cases. Let the sequences
\[
0<s_{1}(q)<s_{2}(q)<s_{3}(q)<\dots \;\mbox{ and }\quad 0<c_{1}(q)<c_{2}(q)<c_{3}(q)<\dots
\]
denote all positive zeros of $\mathcal{S}_{q}$ and $\mathcal{C}_{q}$, respectively.
Then there is an orthogonality relation for $P_{n}(x;q)$ in terms of the function $\mathcal{S}_{q}$ and its zeros only.

\begin{prop}\label{prop:OG_rel_sine}
 The orthogonality relation
 \[
  (1-q)P_{n}(0;q)P_{m}(0;q)-\!\!\sum_{k\in\mathbb{Z}\setminus\{0\}}\frac{\mathcal{S}_{q}(qs_{k}(q))}{s_{k}(q)\mathcal{S}_{q}'(s_{k}(q))}
  P_{n}\left(q^{1/2}s_{k}(q);q\right)P_{m}\left(q^{1/2}s_{k}(q);q\right)=\delta_{mn}
 \]
holds for all $m,n\in\mathbb{Z}_{+}$, where $s_{-k}(q)=-s_{k}(q)$.
\end{prop}

\begin{proof}
 Consider the case $t=0$ in the original parametrization of N-extremal measures (\ref{eq:Zero_t}), (\ref{eq:mu_t}), and (\ref{eq:rho}).
 The corresponding measure is supported by the set of zeros of $D(z;q)=q^{1/2}\mathcal{S}_{q}(q^{-1/2}z)$, i.e., on the set
 $$\{0\}\cup\left\{\pm q^{1/2}s_{k}(q) \mid k\in\mathbb{N}\right\}.$$
 Further, for the function $\rho$, one has
 \[
  \rho\left(q^{1/2}s_{k}(q)\right)=\frac{C(q^{1/2}s_{k}(q);q)}{D'(q^{1/2}s_{k}(q);q)}=\frac{\mathcal{C}_{q}(q^{1/2}s_{k}(q))}{\mathcal{S}_{q}'(s_{k}(q))}=-\frac{\mathcal{S}_{q}(qs_{k}(q))}{s_{k}(q)\mathcal{S}_{q}'(s_{k}(q))}
 \]
 where the last equality holds true due to the first relation of (\ref{eq:qdif_SqCq}). Clearly, the expression for $\rho\left(q^{1/2}s_{k}(q)\right)$ is valid for all nonzero $k\in\mathbb{Z}$, with the convention 
 $s_{-k}(q)=-s_{k}(q)$. Further, one has
 \[
  \rho(0)=\frac{\mathcal{C}_{q}(0)}{\mathcal{S}_{q}'(0)}=1-q,
 \]
 as it follows from the very definition of $\mathcal{S}_{q}$ and $\mathcal{C}_{q}$ (\ref{eq:def_Sq_Cq}). The rest then follows from the general relation (\ref{eq:og_rel_Nextremal_general}).
\end{proof}

One more orthogonality relation for $P_{n}(x;q)$, this time in terms of the function $\mathcal{C}_{q}$ and its zeros 
only, is as follows.

\begin{prop}
 The orthogonality relation
 \[
  -\!\!\sum_{k\in\mathbb{Z}\setminus\{0\}}\frac{\mathcal{C}_{q}(qc_{k}(q))}{c_{k}(q)\mathcal{C}_{q}'(c_{k}(q))}
  P_{n}\left(q^{1/2}c_{k}(q);q\right)P_{m}\left(q^{1/2}c_{k}(q);q\right)=\delta_{mn}
 \]
holds for all $m,n\in\mathbb{Z}_{+}$, where $c_{-k}(q)=-c_{k}(q)$.
\end{prop}

\begin{proof}
 The proof proceeds in the same way as the proof of Proposition \ref{prop:OG_rel_sine}. Nevertheless, we provide the details for the sake of completeness. 

 Consider the case $t=\infty$ in the original parametrization of N-extremal measures (\ref{eq:Zero_t}), (\ref{eq:mu_t}), and (\ref{eq:rho}).
 The corresponding measure is supported by the set of zeros of $B(z;q)=-\mathcal{C}_{q}(q^{-1/2}z)$, i.e., on the set
 $$\left\{\pm q^{1/2}c_{k}(q) \mid k\in\mathbb{N}\right\}.$$
 For the function $\rho$, one has
 \[
  \rho\left(q^{1/2}c_{k}(q)\right)=\frac{A(q^{1/2}c_{k}(q);q)}{B'(q^{1/2}c_{k}(q);q)}=\frac{-q^{1/2}\mathcal{S}_{q}(q^{1/2}c_{k}(q))}{\mathcal{C}_{q}'(c_{k}(q))}=-\frac{\mathcal{C}_{q}(qc_{k}(q))}{c_{k}(q)\mathcal{C}_{q}'(c_{k}(q))}
 \]
 where the last equality holds true due to the second relation of (\ref{eq:qdif_SqCq}). Clearly, the expression for $\rho\left(q^{1/2}c_{k}(q)\right)$ is valid for all nonzero $k\in\mathbb{Z}$, with the convention 
 $c_{-k}(q)=-c_{k}(q)$. The orthogonality relation is then a consequence of (\ref{eq:og_rel_Nextremal_general}).
\end{proof}

\begin{rem} Apart from $N$-extremal measures, there is another exceptional family of measures of orthogonality $\mu_{\varphi}$ 
corresponding to constant Pick functions $\varphi(z)=\beta+i\gamma$, with $\beta\in\mathbb{R}$ and $\gamma>0$. Let us denote these measures by $\mu_{\beta,\gamma}$. 
It turns out that $\mu_{\beta,\gamma}$ is an absolutely continuous measure supported on $\mathbb{R}$ with
the density
\[
\frac{\mbox{d}\mu_{\beta,\gamma}}{\mbox{d}x}=\frac{\gamma}{\pi}\big(\big(\beta B(x)-D(x)\big)^{2}+\gamma^{2}B(x)^{2}\big)^{-1},
\]
see \cite{BergValent}. In particular, the measure $\mu_{0,q^{1/4}}$ has the density
\[
 \frac{\mbox{d}\mu_{0,q^{1/4}}}{\mbox{d}x}=\frac{1}{q^{1/4}\pi}\left|\mathcal{E}_{q}(iq^{-1/2}x)\right|^{-2},
\]
which one verifies by using formulas (\ref{eq:Nevan_ABCD}) and (\ref{eq:qEuler_id}). This was already observed by Chen and 
Ismail in \cite[Thm.~3.4]{ChenIsmail}. Let us point out that, with the aid of (\ref{eq:abs_val_Eq}), the density of $\mu_{0,q^{1/4}}$ can be written in a more
explicit form in terms of ${}_{2}\phi_{2}$. Namely,
\[
 \frac{\mbox{d}\mu_{0,q^{1/4}}}{\mbox{d}x}=\left(\pi q^{1/4}{}_{2}\phi_{2}\left[\begin{matrix}
                       0, & q^{1/2} \\
                       -q^{1/2}, & -q
                      \end{matrix}\ ;q, -q^{-1/2}x^{2} \right]\right)^{-1}\!.
\]
\end{rem}

\section*{Acknowledgements}
The author wishes to acknowledge gratefully the partial support from grant No. GA13-11058S of the Czech Science Foundation.


\begin{thebibliography}{8}

\bibitem{Akhiezer} N.~I.~Akhiezer, The Classical Moment Problem
and Some Related Questions in Analysis, Oliver~\&~Boyd, Edinburgh, 1965.

\bibitem{Al-SalamIsmail} W.~A.~Al-Salam, M.~E.~H.~Ismail, Orthogonal polynomials 
associated with the Rogers-Ramanujan continued fraction, Pac. J.~Math. 104 (1983) 269--283.

\bibitem{Andrews70} G.~E.~Andrews, A Polynomial Identity Which Implies the Rogers-Ramanujan Identities,
Scripta Math. 28 (1970) 297--305.

\bibitem{Andrews04} G.~E.~Andrews, Fibonacci Numbers and Rogers-Ramanujan identities, 
 Fib. Quart. 42 (2004) 3--19.	
 
\bibitem{Andrews_etal} G.~E.~Andrews, A.~Knopfmacher, P.~Paule, An infinite family of Engel expansions of
Rogers-Ramanujan type, Adv. Appl. Math. 25 (2000) 2--11.

\bibitem{Atakishiyev} N.~M.~Atakishiyev, On a one-parameter family of $q$-exponential functions,
J.~Phys.~A 29 (1996) L223--L227.

\bibitem{Berg95} C.~Berg, Indeterminate moment problems and the theory of entire functions, 
J.~Comput. Appl. Math. 65 (1995) 27--55.

\bibitem{BergValent} C.~Berg, G.~Valent, The Nevanlinna parametrization for some indeterminate
Stieltjes moment problems associated with birth and death processes, Methods Appl. Anal. 1
(1994) 169--209.

\bibitem{BerkovichPaule} A.~Berkovich, P.~Paule, Variants of the Andrews-Gordon identities,
Ramanujan J. 5 (2001) 391--404.

\bibitem{BuchwalterCassier} H.~Buchwalter, G.~Cassier, La param\'etrisation de Nevanlinna dans le probl\`eme des moments de
Hamburger, Expo. Math. 2 (1984) 155--178.

\bibitem{BustozCardoso} J.~Bustoz, J.~L.~Cardoso, Basic analog to Fourier series on a $q$-linear grid,
J.~Approx. Theor. 112 (2001) 134--157.

\bibitem{Carlitz74} L.~Carlitz, Fibonacci notes 3: $q$-Fibonacci numbers, 
Fib. Quart. 12 (1974) 317--322.

\bibitem{Carlitz75} L.~Carlitz, Fibonacci notes 4: $q$-Fibonacci polynomials, 
Fib. Quart. 13 (1975) 97--102.

\bibitem{ChenIsmail} Y.~Chen, M.~E.~H.~Ismail, Some indeterminate moment problems and Freud-like weights,
Constr. Approx. 14 (1998) 439--458.

\bibitem{Chihara} T.~S.~Chihara, An Introduction to Orthogonal Polynomials, 
Gordon and Breach Science Publishers, New York, 1978.

\bibitem{Cigler03} J.~Cigler, $q$-Fibonacci polynomials, Fib. Quart. 41 (2003) 31--40.

\bibitem{Cigler04} J.~Cigler, $q$-Fibonacci polynomials and the Rogers-Ramanujan identities, Ann. Comb. 8 (2004) 269--285.

\bibitem{GasperRahman} G.~Gasper, M.~Rahman, Basic Hypergeometric Series, Cambridge University Press, Cambridge, 2004.

\bibitem{Ismail-book} M.~E.~H.~Ismail, Classical and Quantum Orthogonal Polynomials in One Variable, 
Cambridge University Press, Cambridge, 2005.

\bibitem{Ismail05} M.~E.~H.~Ismail, Asymptotics of q-orthogonal polynomials and a q-Airy function, 
Int. Math. Res. Not. 18 (2005) 1063--1088.

\bibitem{Ismail09} M.~E.~H.~Ismail, One parameter generalizations of the Fibonacci and Lucas numbers, Fib. Quart. 47 (2009) 167--179.

\bibitem{IsmailMasson} M.~E.~H.~Ismail, D.~R.~Masson, $q$-Hermite polynomials, biorthogonal rational functions and $q$-beta integrals, 
Trans. Amer. Math. Soc. 346 (1994) 61--116.

\bibitem{Ismail_etal} M.~E.~H.~Ismail, H.~Prodinger, D.~Stanton, Schur's determinants and partition theorems,
 S{\' e}m. Lothar. Combin. 44 (2000) 10 pages.

\bibitem{IsmailZhang} M.~E.~H.~Ismail, R.~Zhang, Diagonalization of certain integral operators,
Adv. Math. 109 (1994) 1--33.

\bibitem{KoelinkSwarttouw} H.~T.~Koelink, R.~F.~Swarttouw, On the zeros of the Hahn-Exton $q$-Bessel function and associated
$q$-Lommel polynomials, J.~Math. Anal. Appl. 186 (1994) 690--710.

\bibitem{KoornwinderSwarttouw} T.~H.~Koornwinder, R.~F.~Swarttouw, On $q$-analogues of the Fourier and Hankel transforms,
Trans. Amer. Math. Soc. 333 (1992) 445--461.

\bibitem{MacMahonII} P.~A.~MacMahon, Combinatory Analysis, Volume II,
Cambridge University Press, Cambridge, 1918.

\bibitem{NelsonGartley} C.~A.~Nelson, M.~G.~Gartley, On the zeros of the $q$-analogue exponential function,
J.~Phys.~A 27 (1994) 3857--3881.

\bibitem{Olver} F.~W.~J.~Olver, Asymptotics and special functions, 
A.~K.~Peters Ltd., Wellesley, 1997.

\bibitem{Rahman} M.~Rahman, The $q$-exponential functions, old and new, in:
A.~N.~Sissakian, G.~S.~Pogosyan, (Eds.), Proc. Internat. Workshop Finite Dimensional Integrable Systems, Joint Institute
for Nuclear Research, Dubna, 1995, pp. 161--170.

\bibitem{Schur} I.~Schur, Ein Beitrag zur Additiven Zahlentheorie und zur Theorie Kettenbr\"{u}che, 
S.-B. Preuss, Akad. Wiss. Phys.-Math. Kl. (1917) 302--321.

\bibitem{ShohatTamarkin} J.~A.~Shohat, J.~D.~Tamarkin, The Problem of Moments, 
American Mathematical Society, New York, 1943.

\bibitem{StampachStovicek11} F.~\v{S}tampach, P.~\v{S}\v{t}ov\'\i\v{c}ek, On the eigenvalue problem for a particular class of finite Jacobi matrices, 
Linear Algebra Appl. 434 (2011) 1336--1353.

\bibitem{StampachStovicekLAA13b} F.~\v{S}tampach, P.~\v{S}\v{t}ov\'\i\v{c}ek, Special functions and spectrum of Jacobi matrices, Linear Algebra Appl. 464 (2015) 38--61.

\bibitem{StampachStovicek_Coulomb} F.~\v{S}tampach, P.~\v{S}\v{t}ov\'\i\v{c}ek, Orthogonal polynomials associated with Coulomb wave functions, 
J. Math. Anal. Appl. 419 (2014) 231--254.

\bibitem{Suslov-book} S.~K.~Suslov, An Introduction to Basic Fourier Series, 
Kluwer Academic Publishers, Dordrecht, 2003.


\end{thebibliography}
\end{document}